\documentclass[11pt]{article} 
\usepackage{amsmath,amssymb,amsthm}  

\allowdisplaybreaks

\begin{document}

\def\fl#1{\left\lfloor#1\right\rfloor}
\def\cl#1{\left\lceil#1\right\rceil}
\def\ang#1{\left\langle#1\right\rangle}
\def\stf#1#2{\left[#1\atop#2\right]} 
\def\sts#1#2{\left\{#1\atop#2\right\}}
\def\eul#1#2{\left\langle#1\atop#2\right\rangle}
\def\N{\mathbb N}
\def\Z{\mathbb Z}
\def\R{\mathbb R}
\def\C{\mathbb C}

\newtheorem{theorem}{Theorem}
\newtheorem{Prop}{Proposition}
\newtheorem{Cor}{Corollary}
\newtheorem{Lem}{Lemma}

\newenvironment{Rem}{\begin{trivlist} \item[\hskip \labelsep{\it
Remark.}]\setlength{\parindent}{0pt}}{\end{trivlist}}

\title{Sylvester sums on the Frobenius set in arithmetic progression with initial gaps
}
 
\author{
Takao Komatsu 
\\
\small Department of Mathematical Sciences, School of Science\\[-0.8ex]
\small Zhejiang Sci-Tech University\\[-0.8ex]
\small Hangzhou 310018 China\\[-0.8ex]
\small \texttt{komatsu@zstu.edu.cn}
}

\date{
\small MR Subject Classifications: Primary 11D07; Secondary 05A15, 05A17, 05A19, 11B68, 11D04, 11P81 
}

\maketitle
 
\begin{abstract} 
Let $a_1,a_2,\dots,a_k$ be positive integers with $\gcd(a_1,a_2,\dots,a_k)=1$. Frobenius number is the largest positive integer that is NOT representable in terms of $a_1,a_2,\dots,a_k$. When $k\ge 3$, there is no explicit formula in general, but some formulae may exist for special sequences $a_1,a_2,\dots,a_k$, including, those forming arithmetic progressions and their modifications.    
In this paper we give explicit formulae for the sum of nonrepresentable positive integers (Sylvester sum) as well as Frobenius numbers and the number of nonrepresentable positive integers (Sylverster number) for $a_1,a_2,\dots,a_k$ forming arithmetic progressions with initial gaps. 
\\
{\bf Keywords:} Frobenius problem, Frobenius numbers, Sylvester numbers, Sylvester sums, arithmetic sequences      
\end{abstract}

\section{Introduction}  

Let $a_1,\dots,a_k$ be positive integers with $\gcd(a_1,\dots,a_k)=1$. It is well-known that all sufficiently large integers can be represented as a nonnegative integer combination of $a_1,\dots,a_k$. Then it is important to determine the largest positive integer that is not representable as a nonnegative integer combination of given positive integers that are coprime. Such a problem is known as the {\it Frobenius Problem} and this largest positive integer is denoted by $g(a_1,\dots,a_k)$ and called the {\it Frobenius number} (see \cite{ra05} for general references)\footnote{Some other symbols have also been used by different backgrounds and authors. The symbols used in this paper are mainly based on literature such as \cite{ra05,se77,tr08}.}. This problem has been also known as the Coin Exchange Problem, Postage Stamp Problem or Chicken McNugget Problem, so has a long history.   
Together with the Frobenius numbers, the number of positive integers with no nonnegative integer representation by $a_1,\dots,a_k$ has also been studied for a long time. This number is sometimes called the {\it Sylvester number} (or {\it genus} in numerical semigroup) and denoted by $n(a_1,\dots,a_k)$.   

According to Sylvester, for positive integers $a$ and $b$ with $\gcd(a,b)=1$,  
\begin{align*}
g(a,b)&=(a-1)(b-1)-1\quad{\rm \cite{sy1884}}\,,\\
n(a,b)&=\frac{1}{2}(a-1)(b-1)\quad{\rm \cite{sy1882}}\,. 
\end{align*}

There are many kinds of problems related to the Frobenius problem. The problems for the number of solutions (e.g., \cite{tr00}), and the sum of integer powers of values the gaps in numerical semigroups (e.g., \cite{bs93,fr07,fks}) are popular. In \cite{Moree14}  the various results within the cyclotomic polynomial and numerical semigroup communities are unified.  
One of other famous problems is about the so-called {\it Sylvester sums} 
$$
s(a_1,\dots,a_k):=\sum_{n\in{\rm NR}(a_1,\dots,a_k)}n 
$$ 
(see, e.g., \cite[\S 5.5]{ra05}, \cite{tu06} and references therein), where ${\rm NR}(a_1,\dots,a_k)$ denotes the set of positive integers without nonnegative integer representation by $a_1,\dots,a_k$. This is exactly the set of gaps in numerical semigroup.  
It is harder to obtain the Sylvester number than the Frobenius number, and even harder to obtain the Sylvester sum. Finally, long time after Sylvester, 
Brown and Shiue \cite{bs93} found the exact value for positive integers $a$ and $b$ with $\gcd(a,b)=1$,  
\begin{equation}
s(a,b)=\frac{1}{12}(a-1)(b-1)(2 a b-a-b-1)\,. 
\label{brown}
\end{equation} 
R\o dseth \cite{ro94} generalized Brown and Shiue's result by giving a closed form for $
s_\mu(a,b):=\sum_{n\in{\rm NR}(a,b)}n^\mu$,  
where $\mu$ is a positive integer.

When $k=2$, there exist beautiful closed forms for Frobenius numbers, Sylvester numbers and Sylvester sums, but 
when $k\ge 3$, exact determination of these numbers is extremely difficult.  
The Frobenius number cannot be given by closed formulas of a certain type (Curtis (1990) \cite{cu90}), the problem to determine $g(a_1,\dots,a_k)$ is NP-hard under Turing reduction (see, e.g., Ram\'irez Alfons\'in \cite{ra05}). 
Nevertheless, one convenient formula is found by Johnson \cite{jo60}. One analytic approach to the Frobenius number can be seen in \cite{bgk01,ko03}.

Though closed forms for general case are hopeless for $k\ge 3$, several formulae for Frobenius numbers, Sylvester numbers and Sylvester sums have been considered under special cases.  For example, one of the best expositions for the Frobenius number in three variables can be seen in \cite{tr17}. For general $k\ge 3$, the Frobenius number and the Sylvester number for some special cases are calculated, including arithmetic sequences and geometric-like sequences (e.g., \cite{br42,op08,ro56,se77}). 

In fact, by introducing the Ap\'ery set, it is possible to determine the functions $g(A)$, $n(A)$ and $s(A)$ for the set of positive integers $A:=\{a_1,a_2,\dots,a_k\}$ with $\gcd(a_1,a_2,\dots,a_k)=1$. 
For the set $A=\{a_1,a_2,\dots,a_k\}$ with $\gcd(A)=1$ and $a_1=\min(A)$ we denote by 
$$
{\rm Ape}(A)={\rm Ape}(A,a_1)=\{m_0,m_1,\dots,m_{a_1-1}\}\,, 
$$ 
the {\it Ap\'ery set} of $A$, where $m_i$ is the least positive integer that can be represented by a nonnegative integral linear combination of $a_2,\dots,a_k$, satisfying $m_i\equiv i\pmod{a_1}$ $(1\le i\le a_1-1)$.  
Note that $m_0$ is defined to be $0$. The element $0$ is often excluded because it does not affect the calculation.   

\begin{Lem}  
We have 
\begin{align*}
g(a_1,a_2,\dots,a_k)&=\left(\max_{1\le i\le a_1-1}m_i\right)-a_1\,,\quad{\rm \cite{bs62}}\\ 
n(a_1,a_2,\dots,a_k)&=\frac{1}{a_1}\sum_{i=1}^{a_1-1}m_i-\frac{a_1-1}{2}\,,\quad{\rm \cite{se77}}\\ 
s(a_1,a_2,\dots,a_k)&=\frac{1}{2 a_1}\sum_{i=1}^{a_1-1}m_i^2-\frac{1}{2}\sum_{i=1}^{a_1-1}m_i+\frac{a_1^2-1}{12}\,.\quad{\rm \cite{tr08}}
\end{align*}
\label{lem1} 
\end{Lem} 

\noindent 
Third formula appeared with a typo in \cite{tr08}, and it has been corrected in \cite{pu18,tr17b}.
Recently, we study the weighted sums and weighted power sums. When $k=2$, a general formula can be expressed in terms of the Apostol-Bernoulli numbers \cite{KZ0}. For general $k$, we can have a formula by using Eulerian numbers \cite{KZ}. In \cite{ko22}, a more general formula including $n(A)$ and $s(A)$ is given by using Bernoulli numbers.     
\bigskip 

As a more general case than the arithmetic sequence, the sequence $a_1=a$, $a_2=a+d$, $\dots$, $a_k=a+(k-1)d$, $a_{k+1}=a+K d$ with $K>k$ has been studied (\cite{se77}). This is one typical case of the so-called {\it almost arithmetic sequence}. In this case, there is an additional term after some gap. As special cases, the Frobenius numbers and the Sylvester number are give for the sequences $a,a+1,a+2,a+4$; $a,a+1,a+2,a+5$; $a,a+1,a+2,a+6$ and so on (\cite{dm64,se77}). Namely, a gap appears in the last. in \cite{si74}, the Frobenius numbers of various cases are expressed in which the $a_i$'s lie in an arithmetic progression, but the results are incomplete. In \cite{ro79}, when the $a_i$ form an almost arithmetic sequence, by considering the Ap\'ery set, algorithms to determine the Sylvester number and sum are given. But, in this paper the computations rely on ${\rm Ape}(A,a+(K+1)d)$ instead of ${\rm Ape}(A,a)$. In \cite{PS90}, the authors gave an alternative description of the Ap\'ery set of the first element in the arithmetic sequence. Their aim was different: they wanted to describe the minimal presentation of the semigroup. 
The approaches in \cite{ro79,PS90} may apply to any almost arithmetic sequence, but they both have an extra burden: the require the pre-computation of a couple of constants depending on the sequence.  Some other applications for the case of almost arithmetic sequences can be found in Section 4 of \cite{RR09}.     

In this paper, we do not study only the Frobenius numbers, but also the Sylvester number and the Sylvester sum where $a_1=a$, $a_2=a+(K+1)d$, $a_3=a+(K+2)d$, $\dots$, $a_{k-K+1}=a+k d$ with $d>0$, $\gcd(a,d)=1$ and $k\ge K+1\ge 2$. Namely, a gap appears in the first. As special cases, we yield these numbers and sums explicitly for the sequences $a,a+2,a+3,a+4$; $a,a+3,a+4,a+5$; $a,a+4,a+5,a+6$; $a,a+5,a+6,a+7$ and so on.

\section{Main result}

For positive integers $a$ and $d$ with $a\ge 2$ and $\gcd(a,d)=1$, consider the sequence $a,a+(K+1)d,a+(K+2)d,\dots,a+k d$. 
Note that the last $(k-K)$ terms form an arithmetic sequence, and there is a gap in the first part. Determine nonnegative integers $q$ and $r$ as  
\begin{equation}
a+K=q k+r,\quad 0\le r<k\,. 
\label{eq:qr}
\end{equation} 
Namely, $q=\fl{(a+K)/k}$ and $r=a+K-\fl{(a+K)/k}k$. 

In this section, we assume that $1\le K\le(k-1)/2$, that is, the gaps between the first term and the second term are not so big. The case when $K>(k-1)/2$ is discussed in the next section. In addition to this condition,  
we shall discuss the two cases separately: $r>K$ or $r\le K$. 

\noindent 
{\bf Case 1:}  
When $r>K$, all the elements of Ap\'ery set excluding $0\pmod a$ can be determined in the following table. 
{\tiny 
\begin{align*}
&&\,&&\,&&\,&a+(K+1)d&\,&\dots&\,&a+k d\\ 
&2 a+(k+1)d&\,&\dots&\,&2 a+(k+K)d&\,&2 a+(k+K+1)d&\,&\dots&\,&2 a+2 k d\\ 
&3 a+(2 k+1)d&\,&\dots&\,&3 a+(2 k+K)d&\,&3 a+(2 k+K+1)d&\,&\dots&\,&3 a+3 k d\\ 
&\dots&\,&\dots&\,&\dots&\,&\dots&\,&\dots&\,&\dots\\ 
&q a+((q-1)k+1)d&\,&\dots&\,&q a+((q-1)k+K)d&\,&q a+((q-1)k+K+1)d&\,&\dots&\,&q a+q k d\\ 
&(q+1)a+(q k+1)d&\,&\dots&\,&\dots&\,&\!\!\!\!\!\!\!\!(q+1)a+(q k+r)d&\,&&\,&\\
&&\,&&\,&\!\!\!\!\!\!\!\!\widehat{(q+1)a+a d}&\,&&\,&\dots&\,&
\end{align*}
} 
The last line consists of $r-1$ terms excluding $(q+1)a+a d=(q+1)a+(q k+r-K)d$, because it is equal to $0\pmod a$. Note that by $\gcd(a,d)=1$, the set of all the elements in this table forms a complete residue system modulo $a$ excluding $0\pmod a$: 
\begin{multline*}
\{(K+1)d,(K+2)d,\dots,(a-1)d,(a+1)d,\dots,(a+K)d\}\\
=\{1,2,\dots,a-1\}\pmod a\,.
\end{multline*} 
Since $K\le(k-1)/2$, $2 a+(k+1)d$ can be represented by using the elements of the last $(K-k)$ terms of the given sequence, which appear in the first line. So, all terms in this table after $2 a+(k+1)d$ can be also represented. In addition, none of the elements in the table can be represented by subtracting $a$. Therefore, the elements in this table form the Ap\'ery set except $0\pmod a$.    
Hence, 
\begin{align*}  
\sum_{i=1}^{a-1}m_i&=\bigl(1\cdot(k-K)+(2+3+\cdots+q)k+(q+1)(r-1)\bigr)a\\
&\quad +\bigl((K+1)+(K+2)+\cdots+(a-1)+(a+1)+\cdots+(a+K)\bigr)d\\
&=\left(\frac{q(q+1)}{2}k-K+(q+1)(r-1)\right)a+\left(\frac{a-1}{2}+K\right)a d
\end{align*}
and 
\begin{align*}  
&\sum_{i=1}^{a-1}m_i^2\\
&=\bigl(1^2\cdot(k-K)+(2^2+3^2+\cdots+q^2)k+(q+1)^2(r-1)\bigr)a^2\\
&\quad +\bigl((K+1)^2+(K+2)^2+\cdots+(a-1)^2+(a+1)^2+\cdots+(a+K)^2\bigr)d^2\\
&\quad +2 a d\biggl(1\bigl((K+1)+\cdots+k\bigr))+2\bigl((k+1)+\cdots+(2 k)\bigr))\\
&\qquad +3\bigl((2 k+1)+\cdots+(3 k)\bigr))+\cdots+q\bigl(((q-1)k+1)+\cdots+(q k)\bigr))\\
&\qquad +(q+1)\bigl((q k+1)+\cdots+(q k+r)-a\bigr))
\biggr)\\
&=\biggl(\frac{q(q+1)(2 q+1)}{6}k-K+(q+1)^2(r-1)\biggr)a^2\\
&\quad +\biggl(\frac{(a+K)(a+K+1)(2 a+2 K+1)}{6}-\frac{K(K+1)(2 K+1)}{6}-a^2\biggr)d^2\\
&\quad + 2 a d\biggl(\left(\frac{k(k+1)}{2}-\frac{K(K+1)}{2}\right)+2\left(\frac{2 k(2 k+1)}{2}-\frac{k(k+1)}{2}\right)\\
&\qquad +3\left(\frac{3 k(3 k+1)}{2}-\frac{2 k(2 k+1)}{2}\right)  
 +\cdots\\
&\qquad +q\left(\frac{q k(q k+1)}{2}-\frac{(q-1)k((q-1)k+1)}{2}\right)\\
&\qquad +(q+1)\left(\frac{(q k+r)(q k+r+1)}{2}-\frac{q k(q k+1)}{2}-a\right)\biggr)\\ 
&=\biggl(\frac{q(q+1)(2 q+1)}{6}k-K+(q+1)^2(r-1)\biggr)a^2\\
&\quad +\frac{\bigl(6 K(a+K+1)+(a-1)(2 a-1)\bigr)}{6}a d^2\\
&\quad +2 a d\biggl(-\frac{K(K+1)}{2}-\frac{q k(q+1)(2 q k+k+3)}{12}\\
&\qquad +\frac{(q+1)(a+K)(a+K+1)}{2}-a(q+1)\biggr)\,. 
\end{align*}
Therefore, by the third formula in Lemma \ref{lem1} together with $q=(a+K-r)/k$, we have 
\begin{align*}
&s(a,a+(K+1)d,\dots,a+k d)\\
&=\frac{1}{2 a}\sum_{i=1}^{a-1}m_i^2-\frac{1}{2}\sum_{i=1}^{a-1}m_i+\frac{a^2-1}{12}\\
&=\frac{1}{12 k^2}\biggl(  
2 a^4+6(K-1)a^3+\bigl(6 K(K-2)-k(k+6)-6 r^2+6(k+2)r\bigr)a^2\\ 
&\quad +\bigl(2 K^2(K-3)-2K k(k+3)\\
&\qquad +4 r^3-6(K+k+1)r^2+(6 K(k+2)+2 k(k+3))r\bigr)a\\ 
&\quad -k^2+\bigl(2 a^2+3(2 K-1)a+6 K^2+6 K+1\bigr)k^2 d^2\\ 
&\quad +\bigl(
4 a^3+3(4 K-3)a^2\\
&\qquad +\bigl(6 K(2 K-1)-k(k+6)-6 r^2+6(k+2)r\bigr)a\\ 
&\qquad +4 K^3-3 K^2(k-1)-K k(k+3)\\
&\qquad +2 r^3-3(2 K+k+1)r^2+k(6 K+k+3)r
\bigr)k d 
\biggr)\,. 
\end{align*}

\noindent 
{\bf Case 2:} 
Let $0\le r\le K$. Since the term $q a+(q k+r-K)d=q a+a d\equiv 0\pmod a$ is excluded, we have 
\begin{align*}  
\sum_{i=1}^{a-1}m_i&=\bigl(1\cdot(k-K)+(2+3+\cdots+q)k-q+(q+1)r\bigr)a\\ 
&\quad +\left(\frac{(a-1)a}{2}+a K\right)d\\
&=\left(\frac{q(q+1)}{2}k-K-q+(q+1)r\right)a+\left(\frac{a-1}{2}+K\right)a d
\end{align*}
and 
\begin{align*}  
&\sum_{i=1}^{a-1}m_i^2\\
&=\bigl(1^2\cdot(k-K)+(2^2+3^2+\cdots+q^2)k-q^2+(q+1)^2 r\bigr)a^2\\
&\quad +\bigl((K+1)^2+(K+2)^2+\cdots+(a-1)^2+(a+1)^2+\cdots+(a+K)^2\bigr)d^2\\
&\quad +2 a d\biggl(1\bigl((K+1)+\cdots+k\bigr))+2\bigl((k+1)+\cdots+(2 k)\bigr))\\
&\qquad +3\bigl((2 k+1)+\cdots+(3 k)\bigr))+\cdots+q\bigl(((q-1)k+1)+\cdots+(q k)\bigr))-q a\biggr)\\
&\qquad +(q+1)\bigl((q k+1)+\cdots+(q k+r)\bigr))
\biggr)\\
&=\biggl(\frac{q(q+1)(2 q+1)}{6}k-K-q^2+(q+1)^2 r\biggr)a^2\\
&\quad +\frac{\bigl(6 K(a+K+1)+(a-1)(2 a-1)\bigr)}{6}a d^2\\
&\quad +2 a d\biggl(-\frac{K(K+1)}{2}-\frac{q k(q+1)(2 q k+k+3)}{12}\\
&\qquad +\frac{(q+1)(a+K)(a+K+1)}{2}-a q\biggr)\,. 
\end{align*}
Therefore, by the third formula in Lemma \ref{lem1} together with $q=(a+K-r)/k$, we have 
\begin{align*}
&s(a,a+(K+1)d,\dots,a+k d)\\
&=\frac{1}{12 k^2}\biggl( 
2 a^4+6(K-1)a^3\\
&\quad +\bigl(6 K(K-2)-k(k-6)-6 r^2+6(k+2)r\bigr)a^2\\ 
&\quad +\bigl(2 K^2(K-3)-2 K k(k-3)\\
&\qquad +4 r^3-6(K+k+1)r^2+(6 K(k+2)-2 k(k-3))r\bigr)a\\ 
&\quad -k^2+\bigl(2 a^2+3(2 K-1)a+6 K^2+6 K+1\bigr)k^2 d^2\\ 
&\quad +\bigl( 
4 a^3+3(4 K-3)a^2\\
&\qquad +\bigl(6 K(2 K-1)-k(k-6)-6 r^2+6(k+2)r\bigr)a\\
&\qquad +4 K^3-3 K^2(k-1)-K k(k+3)\\
&\qquad +2 r^3-3(2 K+k+1)r^2+k(6 K+k+3)r
\bigr)k d 
\biggr)\,. 
\end{align*}

\begin{theorem}  
Let $a,d,K,k$ be positive integers with $a\ge 2$, $\gcd(a,d)=1$ and $K\le(k-1)/2$. Let $r=a+K-\fl{(a+K)/k}k$. If $r>K$, then  
\begin{align*}
&s(a,a+(K+1)d,\dots,a+k d)\\
&=\frac{1}{12 k^2}\biggl(  
2 a^4+6(K-1)a^3+\bigl(6 K(K-2)-k(k+6)-6 r^2+6(k+2)r\bigr)a^2\\ 
&\quad +\bigl(2 K^2(K-3)-2K k(k+3)\\
&\qquad +4 r^3-6(K+k+1)r^2+(6 K(k+2)+2 k(k+3))r\bigr)a\\ 
&\quad -k^2+\bigl(2 a^2+3(2 K-1)a+6 K^2+6 K+1\bigr)k^2 d^2\\ 
&\quad +\bigl(
4 a^3+3(4 K-3)a^2\\
&\qquad +\bigl(6 K(2 K-1)-k(k+6)-6 r^2+6(k+2)r\bigr)a\\ 
&\qquad +4 K^3-3 K^2(k-1)-K k(k+3)\\
&\qquad +2 r^3-3(2 K+k+1)r^2+k(6 K+k+3)r
\bigr)k d 
\biggr)\,.  
\end{align*} 
If $0\le r\le K$, then 
\begin{align*}
&s(a,a+(K+1)d,\dots,a+k d)\\
&=\frac{1}{12 k^2}\biggl( 
2 a^4+6(K-1)a^3\\
&\quad +\bigl(6 K(K-2)-k(k-6)-6 r^2+6(k+2)r\bigr)a^2\\ 
&\quad +\bigl(2 K^2(K-3)-2 K k(k-3)\\
&\qquad +4 r^3-6(K+k+1)r^2+(6 K(k+2)-2 k(k-3))r\bigr)a\\ 
&\quad -k^2+\bigl(2 a^2+3(2 K-1)a+6 K^2+6 K+1\bigr)k^2 d^2\\ 
&\quad +\bigl( 
4 a^3+3(4 K-3)a^2\\
&\qquad +\bigl(6 K(2 K-1)-k(k-6)-6 r^2+6(k+2)r\bigr)a\\
&\qquad +4 K^3-3 K^2(k-1)-K k(k+3)\\
&\qquad +2 r^3-3(2 K+k+1)r^2+k(6 K+k+3)r
\bigr)k d 
\biggr)\,. 
\end{align*} 
\label{th:aKk}
\end{theorem}

By applying the first formula in Lemma \ref{lem1}, we can obtain the Frobenius number of the almost arithmetic sequence with initial gaps. Here, integers $a,d,K,k,r$ are determined as in Theorem \ref{th:aKk}.  

\begin{theorem}  
Under the same conditions as in Theorem \ref{th:aKk}, we have 
\begin{align*}
g(a,a+(K+1)d,\dots,a+k d)&=\begin{cases} 
\frac{a(a+K-r)}{k}+(a+K)d&\text{if $r>0$}\\ 
\frac{a(a+K-k)}{k}+(a+K)d&\text{if $r=0$}
\end{cases}\\
&=\left(\cl{\frac{a+K}{k}}-1\right)a+(a+K)d\,. 
\end{align*} 
\label{cor:aKk-fn} 
\end{theorem}
\begin{proof}
If $r>0$, $(a+K)/k$ is not an integer.  Then by $q=(a+K-r)/k$, 
\begin{align*} 
g(a,a+(K+1)d,\dots,a+k d)&=(q+1)a+(q k+r)d-a\\
&=\frac{a(a+K-r)}{k}+(a+K)d\\
&=\fl{\frac{a+K}{k}}a+(a+K)d\\
&=\left(\cl{\frac{a+K}{k}}-1\right)a+(a+K)d\,. 
\end{align*}
If $r=0$, $(a+K)/k=q$ is an integer.  Then 
\begin{align*} 
g(a,a+(K+1)d,\dots,a+k d)&=q a+(q k+r)d-a\\
&=\frac{a(a+K-k)}{k}+(a+K)d\,. 
\end{align*}   
\end{proof}

By applying the second formula in Lemma \ref{lem1}, we have the Sylvester number of the almost arithmetic sequence with initial gaps.  

\begin{theorem}   
Under the same conditions as in Theorem \ref{th:aKk}, if $r>K$, then  
\begin{align*}
&n(a,a+(K+1)d,\dots,a+k d)\\
&=\dfrac{a^2+2(K-1)a+(a+2K-1)k d+K(K-k-2)-k-r(r-k-2)}{2 k}
\end{align*} 
and if $0\le r\le K$, then 
\begin{align*}
&n(a,a+(K+1)d,\dots,a+k d)\\
&=\dfrac{a^2+2(K-1)a+(a+2K-1)k d+K(K-k-2)+k-r(r-k-2)}{2 k}\,. 
\end{align*} 
\label{cor:aKk-sn} 
\end{theorem}
\begin{proof}  
If $r>K$, by the second formula in Lemma \ref{lem1} together with $q=(a+K-r)/k$, we have 
\begin{align*}  
&n(a,a+(K+1)d,\dots,a+k d)\\
&=\frac{q(q+1)}{2}k-K+(q+1)(r-1)+\left(\frac{a-1}{2}+K\right)d-\frac{a-1}{2}\\
&=\frac{a^2+2(K-1)a+(a+2K-1)k d+K(K-k-2)-k-r(r-k-2)}{2 k}\,. 
\end{align*}
Other cases are proved similarly and omitted.  
\end{proof}

\subsection{Examples} 

Consider the sequence $11, 23, 27, 31$.  Then, $a=11$, $d=4$, $K=2$, $k=5$, $q=2$ and $r=3$. By Theorem \ref{th:aKk}, we have $s(11,23,27,31)=1149$. Indeed, 
\begin{align*}  
&s(11,23,27,31)\\
&=1 + 2 + 3 + 4 + 5 + 6 + 7 + 8 + 9 + 10 + 12 + 13 + 14 + 15 + 16 + 17\\
&\quad + 18 + 19 + 20 + 21 + 24 + 25 + 26 + 28 + 29 + 30 + 32 + 35 + 36\\
&\quad + 37 + 39 + 40 + 41 + 43 + 47 + 48 + 51 + 52 + 59 + 63 + 70 + 74\\
&=1149\,. 
\end{align*}

Consider the sequence $13,22,25,28$.  Then, $a=13$, $d=3$, $K=2$, $k=5$ and $q=3$. By Theorem \ref{th:aKk}, we have 
\begin{align*}
\sum_{i=1}^{a-1}m_i&=22 + 25 + 28 + 44 + 47 + 50 + 53 + 56 + 72 + 75 + 81 + 84\\
&=637\,,\\
\sum_{i=1}^{a-1}m_i^2&=22^2 + 25^2 + 28^2 + 44^2 + 47^2 + 50^2 + 53^2 + 56^2 + 72^2 + 75^2\\
&\quad +  81^2 + 84^2\\
&=38909 
\end{align*} 
and $s(13,22,25,28)=1192$. 

Consider the sequence $10,22,25,28,31,34,37,40$.  Then, $a=10$, $d=3$, $K=3$, $k=10$, $q=1$ and $r=3$. By Theorem \ref{th:aKk}, we have 
\begin{align*}
\sum_{i=1}^{a-1}m_i&=22 + 25 + 28 + 31 + 34 + 37 + 53 + 56 + 59\\
&=345\,,\\
\sum_{i=1}^{a-1}m_i^2&=22^2 + 25^2 + 28^2 + 31^2 + 34^2 + 37^2 + 53^2 + 56^2 + 59^2\\
&=14805 
\end{align*} 
and $s(10,22,25,28,31,34,37,40)=576$.

\subsection{Special patterns} 

For an integer $a\ge 2$, let us consider the sequence $a,a+2,a+3,a+4$. So, $K=1$, $k=4$ and $d=1$. Nonnegative integers $q$ and $r$ are determined as 
$$
a+1=4 q+r,\quad 0\le r<4\,. 
$$ 
When $r=2,3$, that is $a\equiv 1,2\pmod 4$, by 
\begin{align*}
\sum_{i=1}^{a-1}m_i&=\frac{a\bigl(a-3+4 q^2+2 r+2 q(r+1)\bigr)}{2}\\ 
&=\frac{a\bigl(a(a+8)-(r-3)^2\bigr)}{8}
\end{align*} 
and 
\begin{align*}
&\sum_{i=1}^{a-1}m_i^2\\
&=\frac{1}{6}a\bigl(2 a^2+(8 q^3+6(r+1)q^2+4(3 r-5)q+6 r-21)a\\
&\qquad +64 q^3 +12(4 r+5)q^2+(6 r^2+54 r-4)q+6 r^2+6 r+1\bigr)\\ 
&=\frac{1}{48}a\bigl(a^4+14 a^3-3(r^2-6 r-13)a^2\\
&\qquad +2(r^3-18 r^2+77 r-64)a+(4 r^3-42 r^2+104 r+38)\bigr)\,, 
\end{align*} 
we have 
\begin{align*} 
&s(a,a+2,a+3,a+4)\\
&=\frac{(a-r+5)\bigl(a^3+(r+3)a^2-2(r^2-8 r+8)a-4 r^2+22 r+6\bigr)}{96}\,. 
\end{align*}
When $r=0$, that is $a\equiv 3\pmod 4$, by 
\begin{align*}
\sum_{i=1}^{a-1}m_i&=\frac{a(a-1+4 q^2+2 q)}{2}\\ 
&=\frac{a(a^2+8 a-1)}{8}
\end{align*} 
and 
\begin{align*}
\sum_{i=1}^{a-1}m_i^2
&=\frac{1}{6}a\bigl(2 a^2+(8 q^3+6 q^2-8 q-3)a+64 q^3+60 q^2-4 q+1\bigr)\\ 
&=\frac{a(a^4+14 a^3+63 a^2+40 a+38)}{48}\,, 
\end{align*} 
we have 
\begin{align*} 
s(a,a+2,a+3,a+4)
=\frac{(a+1)(a+5)(a^2+2 a+6)}{96}\,. 
\end{align*}
When $r=1$, that is $a\equiv 0\pmod 4$,  by 
\begin{align*}
\sum_{i=1}^{a-1}m_i&=\frac{a\bigl(a+(2 q+1)^2\bigr)}{2}\\ 
&=\frac{a(a^2+8 a+4)}{8}
\end{align*} 
and 
\begin{align*}
\sum_{i=1}^{a-1}m_i^2
&=\frac{1}{6}a\bigl(2 a^2+(8 q^3+12 q^2+4 q+3)a\\
&\qquad +64 q^3 +108 q^2+56 q+13\bigr)\\ 
&=\frac{a(a^4+14 a^3+78 a^2+136 a+108)}{48}\,, 
\end{align*} 
we have 
\begin{align*} 
s(a,a+2,a+3,a+4)
=\frac{(a+4)(a^3+4 a^2+22 a+24)}{96}\,. 
\end{align*}

In conclusion, we have the following. 

\begin{Cor}    
For $a\ge 2$, we have 
$$ 
s(a,a+2,a+3,a+4)=\begin{cases}
\dfrac{(a+4)(a^3+4 a^2+22 a+24)}{96}&\text{if $a\equiv 0\pmod 4$}\\ 
\dfrac{(a+3)(a^3+5 a^2+8 a+34)}{96}&\text{if $a\equiv 1\pmod 4$}\\ 
\dfrac{(a+2)(a^3+6 a^2+14 a+36)}{96}&\text{if $a\equiv 2\pmod 4$}\\ 
\dfrac{(a+5)(a^3+3 a^2+8 a+6)}{96}&\text{if $a\equiv 3\pmod 4$}\,. 
\end{cases}
$$ 
\label{cor:a234}
\end{Cor}

For example, 
\begin{align*}
s(12,14,15,16)&=432\,,\\
s(13,15,16,17)&=530\,,\\
s(14,16,17,18)&=692\,,\\
s(15,17,18,19)&=870\,.
\end{align*}

From Theorem \ref{cor:aKk-fn}, we have 
$$  
g(a,a+2,a+3,a+4)=\begin{cases}
\dfrac{a^2+4 a+4}{4}&\text{if $a\equiv 0\pmod 4$}\\ 
\dfrac{a^2+3 a+4}{4}&\text{if $a\equiv 1\pmod 4$}\\ 
\dfrac{a^2+2 a+4}{4}&\text{if $a\equiv 2\pmod 4$}\\ 
\dfrac{a^2+a+4}{4}&\text{if $a\equiv 3\pmod 4$}\,. 
\end{cases}
$$ 
Notice that $r=0,1,2,3$ implies that $a\equiv 3,0,1,2\pmod 4$, respectively.  
By using the floor function, we can rewritten as follows.  

\begin{Cor}    
For $a\ge 2$, we have 
$$ 
g(a,a+2,a+3,a+4)=\left(1+\fl{\frac{a}{4}}\right)a+1\,. 
$$ 
\label{cor:g234}
\end{Cor}

From Theorem \ref{cor:aKk-sn}, we have the following. 

\begin{Cor}    
For $a\ge 2$, we have 
$$ 
n(a,a+2,a+3,a+4)=\begin{cases}
\dfrac{a^2+4 a+8}{8}&\text{if $a\equiv 0\pmod 4$}\\ 
\dfrac{a^2+4 a+3}{8}&\text{if $a\equiv 1\pmod 4$}\\ 
\dfrac{a^2+4 a+4}{8}&\text{if $a\equiv 2\pmod 4$}\\ 
\dfrac{a^2+4 a+3}{8}&\text{if $a\equiv 3\pmod 4$}\,. 
\end{cases} 
$$ 
\label{cor:g234}
\end{Cor}

The sequence $a,a+3,a+4,a+5$ also satisfies the condition $1\le K\le(k-1)/2$ as $K=2$, $k=5$ and $d=1$.  

\begin{Cor}    
For $a\ge 3$, we have 
$$ 
s(a,a+3,a+4,a+5)=\begin{cases}
\dfrac{(a+5)(a^3+8 a^2+55 a+90)}{150}&\text{if $a\equiv 0\pmod 5$}\\ 
\dfrac{(a+4)(a^3+9 a^2+35 a+105)}{150}&\text{if $a\equiv 1\pmod 5$}\\ 
\dfrac{(a+3)(a^3+10 a^2+41 a+110)}{150}&\text{if $a\equiv 2\pmod 5$}\\ 
\dfrac{(a+2)(a^3+11 a^2+43 a+105)}{150}&\text{if $a\equiv 3\pmod 5$}\\ 
\dfrac{(a+6)(a^3+7 a^2+41 a+65)}{150}&\text{if $a\equiv 4\pmod 5$}\,. 
\end{cases}
$$ 
\label{cor:a345}
\end{Cor}

\begin{Cor}    
For $a\ge 3$, we have 
$$ 
g(a,a+3,a+4,a+5)=\left(1+\fl{\frac{a+1}{5}}\right)a+2\,. 
$$ 
\label{cor:g345}
\end{Cor}

\begin{Cor}    
For $a\ge 3$, we have 
$$ 
n(a,a+3,a+4,a+5)=\begin{cases}
\dfrac{a^2+7 a+20}{10}&\text{if $a\equiv 0\pmod 5$}\\ 
\dfrac{a^2+7 a+12}{10}&\text{if $a\equiv 1\pmod 5$}\\ 
\dfrac{a^2+7 a+12}{10}&\text{if $a\equiv 2\pmod 5$}\\ 
\dfrac{a^2+7 a+10}{10}&\text{if $a\equiv 3\pmod 5$}\\  
\dfrac{a^2+7 a+16}{10}&\text{if $a\equiv 4\pmod 5$}\,.
\end{cases} 
$$ 
\label{cor:g345}
\end{Cor}

However, the sequence $a,a+4,a+5,a+6$ does not satisfy the condition $1\le K\le(k-1)/2$. In fact, if the gap $K$ becomes bigger compared to $k$, the situation becomes more complicated. This case is discussed in the next section.

\section{Bigger gaps}  

When the gap $K$ is bigger, the situation becomes more complicated.  
Consider the same almost arithmetic sequence $a,a+(K+1)d,a+(K+2)d,\dots,a+k d$ with $a\ge 2$ and $\gcd(a,d)=1$. 
In this section, we treat with the case when $K>(k-1)/2$. Nevertheless, this case cannot be treated in a unified manner. Cases still need to be divided.

\subsection{General case} 

Assume that $(k-1)/2<K\le(2 k-2)/3$. Nonnegative integers $q$ and $r$ are determined as in (\ref{eq:qr}).   
We also assume that $q=\fl{(a+K)/k}\ge 2$.   
In addition to these conditions,   
we shall discuss four cases separately: $r=0$, $1\le r\le k-K-1$, $k-K\le r\le K$ or $K+1\le r<k$.

\noindent 
{\bf Case 1:}  
When $r=a+K-\fl{(a+K)/k}k=0$, all the elements of Ap\'ery set of $a,a+(K+1)d,a+(K+2)d,\dots,a+k d$ excluding $0\pmod  a$ can be determined in the following table. 
{\tiny  
\begin{align*}
&&\,&&\,&&\,&&\,&a+(K+1)d&\,&\dots&\,&a+k d\\
&&\,&&\,&2 a+(2K+2)d&\,&\dots&\,&2 a+(K+k+1)d&\,&\dots&\,&2 a+2 k d\\
&3 a+(2 k+1)d&\,&\dots&\,&3 a+(2K+k+2)d&\,&\dots&\,&3 a+(K+2 k+1)d&\,&\dots&\,&3 a+3 k d\\&\dots&\,&\dots&\,&\dots&\,&\dots&\,&\dots&\,&\dots&\,&\dots\\
&q a+((q-1)k+1)d&\,&\dots&\,&&\,&\!\!\!\!\!\dots&\,&\!\!\!\!\![q a+(q k-K)d]&\,&\dots&\,&q a+q k d\\
&&\,&&\,&\!\!\!\!\!(q+1)a+(q k+k-K+1)d&\,&\!\!\!\!\!\dots&\,&\!\!\!\!\!(q+1)a+(q k+K+1)d&\,&&\,&
\end{align*}
} 
Since any of the terms $2 a+(k+1)d,\dots,2 a+(2 K+1)d$ cannot be expressed in terms of $a,a+(K+1)d,a+(K+2)d,\dots,a+k d$, they cannot be elements of the Ap\'ery set, so do not exist in this table. In addition, the term $q a+(q k-K)d=q a+a d\equiv 0\pmod a$ cannot exist in this table, and the terms $(q+1)a+(q k+1)d,\dots,(q+1)a+(q k+k-K)d$ are also out.   
Note that by $\gcd(a,d)=1$, 
\begin{align*} 
&\{(K+1)d,(K+2)d,\dots,k d,(2 K+2)d,\dots,(a-1)d,\\
&\qquad (a+1)d,\dots,(a+K)d,(a+k+1)d,\dots,(a+2 K+1)d\}\\
&=\{1,2,\dots,a-1\}\pmod a\,.
\end{align*}
It is not difficult to see that all the elements in this table can be represented by $a+(K+1)d,a+(K+2)d,\dots,a+k d$, and that none of the elements can be represented by subtracting $a$. 
Hence,   
\begin{align*} 
&\sum_{i=1}^{a-1}m_i\\
&=\bigl((1+2+\dots+q)k-K-2(2 K+1-k)-q+(q+1)(2 K-k+1)\bigr)a\\ 
&\quad +\bigl((k+1)+\cdots+k+(2 K+2)+\cdots+(a-1)+(a+1)+\cdots+(a+K)\\
&\qquad +(a+k+1)+\cdots+(a+2 K+1)\bigr)d\\
&=\left(\frac{q(q+1)}{2}k-5 K-2+2 k-q+(q+1)(2 K-k+1)\right)a\\
&\quad+\left(\frac{a(a-1)}{2}+(3 K-k+1)a\right)d\\
&=\frac{a\bigl(a^2+(6 K-k)a+K(5 K-7 k)+2 k(k-1)\bigr)}{2 k}
+\left(\frac{a+1}{2}+3 K-k\right)a d
\end{align*} 
and 
\begin{align*} 
&\sum_{i=1}^{a-1}m_i^2\\
&=\bigl((1^2+2^2+\dots+q^2)k-K-2^2(2 K+1-k)-q^2+(q+1)^2(2 K-k+1)\bigr)a^2\\ 
&\quad +\bigl((k+1)^2+\cdots+k^2+(2 K+2)^2+\cdots+(a-1)^2+(a+1)^2+\cdots+(a+K)^2\\
&\qquad +(a+k+1)^2+\cdots+(a+2 K+1)^2\bigr)d^2\\
&\quad +2 a d\biggl(\left(\frac{k(k+1)}{2}-\frac{K(K+1)}{2}\right)+2\left(\frac{2 k(2 k+1)}{2}-\frac{k(k+1)}{2}\right)\\
&\qquad -2\left(\frac{(2 K+1)(2 K+2)}{2}-\frac{k(k+1)}{2}\right)+3\left(\frac{3 k(3 k+1)}{2}-\frac{2 k(2 k+1)}{2}\right)\\
&\qquad +\cdots+q\left(\frac{q k(q k+1)}{2}-\frac{(q k-q)(q k-q+1)}{2}\right)\\
&\qquad -q a+(q+1)\left(\frac{(a+2 K+1)(a+2 K+2)}{2}-\frac{(a+k)(a+k+1)}{2}\right)
\biggr)\\
&=\frac{a^2}{6 k^2}\bigl(2 a^3+3(6 K-k)a^2+(6 K(5 K+3 k))a+14 K^3\\
&\qquad +(21 K-53 k+12)K k+18 k^2(k-1)\bigr)a^2\\
&\quad +\frac{a}{6}\bigl(2 a^2+3(6 K-2 k+1)a+6 K(5 K+7 k)-6 k^2-6 k+13\bigr)d^2\\
&\quad +\frac{a d}{6 k}\bigl(4 a^3+3(12 K-3 k+1)a^2+(6 K(10 K+3 k+7)-19 k^2+9 k+12)a\\
&\qquad +K(28 K^2+39 K+12)-3(9 K^2+15 K+4)k-(7 K-6 k-6)k^2\bigr)\,, 
\end{align*} 
by the third formula in Lemma \ref{lem1}, we have 
\begin{align*}   
&s\bigl(a,a+(K+1)d,a+(K+2)d,\dots,a+k d\bigr)\\
&=\frac{1}{12 k^2}\biggl(2 a^4+6(3 K-k)a^3+\bigl(30 K^2-k(7 k-12)\bigr)a^2\\
&\quad +\bigl(14 K^3+12 k^3-4(8 K+3)k^2+6 K(K+2)k\bigr)a-k^2\\
&\quad +\bigl(2 a^2+3(6 K-2 k+1)a+30 K^2+42 K-6 k^2-6 k+13\bigr)k^2 d^2\\ 
&\quad +\bigl(4 a^3+3(12 K-4 k+1)a^2+(60 K^2+42 K-13 k^2+6 k+12)a\\
&\qquad +6 k^3-(7 K-6)k^2-3(9 K^2+15 K+4)k+K(28 K^2+39 K+12)
\bigr)k d\biggr)\,. 
\end{align*}
In addition, by the first and the second formulae in Lemma \ref{lem1}, we have 
$$
g\bigl(a,a+(K+1)d,a+(K+2)d,\dots,a+k d\bigr)
=\frac{a(a+K)}{k}+(a+2 K+1)d 
$$
and 
\begin{align*}  
&n\bigl(a,a+(K+1)d,a+(K+2)d,\dots,a+k d\bigr)\\
&=\frac{a^2+2(3 K-k)a+K(5 K-7 k)+k(2 k-1)}{2 k}+\left(\frac{a+1}{2}+3 K-k\right)d\,,
\end{align*} 
respectively.  


\noindent 
{\bf Case 2:}  
Assume that $1\le r\le k-K-1$. Since $k-K-1<K$, the $q$-th and the ($q+1$)-th lines are replaced by  
\begin{align*} 
\underbrace{q a+((q-1)k+1)d\,\dots\,q a+(a-1)d}_{k-K+r-1}\,\underbrace{\text{[gap]}}_1\,\underbrace{q a+(a+1)d\,\dots\,q a+q k d}_{K-r} 
\end{align*} 
and 
\begin{multline*} 
\underbrace{(q+1)a+(q k+1)d\,\dots\,(q+1)a+(q k+r)d}_{r}\underbrace{\text{[\phantom{gap}gap\phantom{gap}]}}_{k-K}\\
\underbrace{(q+1)a+(a+k+1)d\,\dots\,(q+1)a+(a+2 K+1)d}_{2 K-k+1}\underbrace{\text{[\phantom{gap}gap\phantom{gap}]}}_{k-K-r-1}
\end{multline*}  
from Case 1, respectively.  Then,  
by 
\begin{align*} 
&\sum_{i=1}^{a-1}m_i\\
&=\biggl(\frac{q(q+1)}{2}k-K-2(2 K+1-k)-q+(q+1)(2 K-k+1+r)\biggr)a\\ 
&\quad +\biggl(\frac{a(a+1)}{2}+(3 K-k)a\biggr)d\\ 
&=\frac{a\bigl(a^2+(6K-k)a+K(5 K-7 k)+2 k(k-1)-r(r+4 K-3 k)\bigr)}{2 k}\\
&\quad +\left(\frac{a+1}{2}+3 K-k\right)a d
\end{align*}
and  
\begin{align*} 
&\sum_{i=1}^{a-1}m_i^2\\
&=\biggl(
\frac{q(q+1)(2 q+1)}{6}k-9 K-4+4 k-q^2+(q+1)^2(2 K-k+1+r)
\biggr)a^2\\ 
&\quad +\left(\frac{(a+1)(2 a+1)}{6}+(3 K-k)a+(K+1)(5 K+2)-k(k+1)\right)a d^2\\
&\quad +2 a d\biggl(-\frac{(K+1)(9 K+4)}{2}+k(k+1)+\frac{q k(q+1)\bigl((4 q-1)k+3\bigr)}{12}\\
&\qquad -q a+(q+1)\biggl(\frac{(a+2 K+1)(a+2 K+2)}{2}-\frac{(a+k)(a+k+1)}{2}\\
&\qquad +\left(q k+\frac{r+1}{2}\right)r\biggr)\biggr)\,,  
\end{align*}
we have 
\begin{align*}   
&s\bigl(a,a+(K+1)d,a+(K+2)d,\dots,a+k d\bigr)\\
&=\frac{1}{12 k^2}\biggl(
2 a^4+6(3 K-k)a^3+\bigl(30 K^2-7 k^2+12 k-6(r+4 K-3 k)r\bigr)a^2\\ 
&\quad +\bigl(K^2(14 K+6 k)-4 K k(8 k-3)+12 k^2(k-1)+4 r^3+6(K-2 k)r^2\\ 
&\qquad -(24 K^2-6 K k-8 k^2+12 k)r\bigr)a -k^2\\ 
&\quad +\bigl(2 a^2+3(6 K-2 k+1)a+6 K(5 K+7)-(6 k^2+6 k-13)\bigr)k^2 d^2\\
&\quad +\bigl(
4 a^3+3(12 K-4 k+1)a^2\\
&\qquad +\bigl(6 K(10 K+7 k)-(13 k^2-6 k-12)-6 r(r+4 K-3 k)\bigr)a\\ 
&\qquad +28 K^3-K^2(27 k-39)-K(7 k^2+45 k-12)+6(k-1)k(k+2)
\bigr)k d
\biggr)\,.  
\end{align*} 
This also holds for $r=0$.

We also have 
$$
g\bigl(a,a+(K+1)d,a+(K+2)d,\dots,a+k d\bigr)
=\frac{a(a+K-r)}{k}+(a+2 K+1)d 
$$ 
and 
\begin{align*}  
&n\bigl(a,a+(K+1)d,a+(K+2)d,\dots,a+k d\bigr)\\
&=\frac{a^2+2(3 K-k)a+K(5 K-7 k)+k(2 k-1)-r(r+4 K-3 k)}{2 k}\\
&\quad +\left(\frac{a+1}{2}+3 K-k\right)d\,. 
\end{align*}


\noindent 
{\bf Case 3:}  
Assume that $k-K\le r\le K$. Then, the $q$-th, the ($q+1$)-th and ($q+2$)-th lines are replaced by 
\begin{align*} 
\underbrace{q a+((q-1)k+1)d\,\dots\,q a+(a-1)d}_{k-K+r-1}\,\underbrace{\text{[gap]}}_1\,\underbrace{q a+(a+1)d\,\dots\,q a+q k d}_{K-r} 
\end{align*} 
\begin{multline*} 
\underbrace{(q+1)a+(q k+1)d\,\dots\,(q+1)a+(q k+r)d}_r\underbrace{\text{[\phantom{gap}gap\phantom{gap}]}}_{k-K}\\
\underbrace{(q+1)a+(a+k+1)d\,\dots\,(q+1)a+(a+k+K-r)d}_{K-r}
\end{multline*}
and 
$$
\underbrace{(q+2)a+((q+1)k+1)d\,\dots\,(q+2)a+(a+2 K+1)d}_{K-k+r+1}\,\underbrace{\text{[\phantom{gap}gap\phantom{gap}]}}_{2 k-K-r-1}
$$   
respectively.  Then,  
by 
\begin{align*} 
&\sum_{i=1}^{a-1}m_i\\
&=\biggl(\frac{q(q+1)}{2}k-K-2(2 K+1-k)-q+(q+1)K+(q+2)(K-k+r+1)\biggr)a\\ 
&\quad +\biggl(\frac{a(a+1)}{2}+(3 K-k)a\biggr)d\\ 
&=\frac{a\bigl(a^2+(6K-k)a+(K-r)(5 K-5 k+r)\bigr)}{2 k}
 +\left(\frac{a+1}{2}+3 K-k\right)a d
\end{align*}
and  
\begin{align*} 
&\sum_{i=1}^{a-1}m_i^2\\
&=\biggl(
\frac{q(q+1)(2 q+1)}{6}k-9 K-4+4 k-q^2+(q+1)^2 K+(q+2)^2(K-k+r+1)
\biggr)a^2\\ 
&\quad +\left(\frac{(a+1)(2 a+1)}{6}+(3 K-k)a+(K+1)(5 K+2)-k(k+1)\right)a d^2\\
&\quad +2 a d\biggl(-\frac{(K+1)(9 K+4)}{2}+k(k+1)+\frac{q k(q+1)\bigl((4 q-1)k+3\bigr)}{12}\\
&\qquad -q a+(q+1)\biggl(\frac{(a+K+k-r)(a+K+k-r+1)}{2}-\frac{(a+k)(a+k+1)}{2}\\
&\qquad\quad +\left(q k+\frac{r+1}{2}\right)r\biggr)\\
&\qquad +(q+2)\left((q+1)k(K-k+r+1)+\frac{(K-k+r+1)(K-k+r+1)}{2}\right)
\biggr)\,,  
\end{align*}
we have 
\begin{align*}   
&s\bigl(a,a+(K+1)d,a+(K+2)d,\dots,a+k d\bigr)\\
&=\frac{1}{12 k^2}\biggl(
2 a^4+6(3 K-k)a^3+\bigl(30 K^2+12 K k-19 k^2+24 k-6(r+4 K-5 k)r\bigr)a^2\\ 
&\quad +\bigl(K^2(14 K+18 k)-8 K k(4 k-3)+4 r^3+6(K-4 k)r^2\\ 
&\qquad -(24 K^2-6 K k-32 k^2+24 k)r\bigr)a -k^2\\ 
&\quad +\bigl(2 a^2+3(6 K-2 k+1)a+6 K(5 K+7)-6 k^2-6 k+13\bigr)k^2 d^2\\
&\quad +\bigl(
4 a^3+3(12 K-4 k+1)a^2\\
&\qquad +\bigl(6 K(10 K+2 k+7)-(25 k^2-18 k-12)-6 r(r+4 K-5 k)\bigr)a\\ 
&\qquad +28 K^3-K^2(9 k-39)-K(19 k^2+15 k-12)+2 r^3-3 r^2(2 K+3 k+1)\\
&\qquad +r\bigl(-K(24 K-18 k+36)+(19 k^2+15 k-12)\bigr)
\bigr)k d
\biggr)\,.  
\end{align*} 

We also have 
$$
g\bigl(a,a+(K+1)d,a+(K+2)d,\dots,a+k d\bigr)
=\frac{a(a+K-r+k)}{k}+(a+2 K+1)d 
$$ 
and 
\begin{align*}  
&n\bigl(a,a+(K+1)d,a+(K+2)d,\dots,a+k d\bigr)\\
&=\frac{a^2+2(3 K-k)a+(K-r)(5 K-5 k+r)+k}{2 k}\\
&\quad +\left(\frac{a+1}{2}+3 K-k\right)d\,. 
\end{align*}


\noindent 
{\bf Case 4:}  
Finally, assume that $K+1\le r<k$. Then, the $q$-th line consists of no gaps. The ($q+1$)-th and ($q+2$)-th lines are replaced by  
\begin{multline*} 
\underbrace{(q+1)a+(q k+1)d\,\dots\,(q+1)a+(q k+r-K-1)d}_{r-K-1}\,\underbrace{\text{[gap]}}_1\,\\
\underbrace{(q+1)a+(a+1)d\,\dots\,(q+1)a+(a+K)d}_{K}\,\underbrace{\text{[\phantom{g}gap\phantom{g}]}}_{k-r}
\end{multline*}
and 
$$
\underbrace{\text{[\phantom{g}gap\phantom{g}]}}_{r-K}\,\underbrace{(q+2)a+(a+k+1)d\,\dots\,(q+2)a+(a+2 K+1)d}_{2 K-k+1}\,\underbrace{\text{[\phantom{gap}gap\phantom{gap}]}}_{2 k-K-r-1}
$$   
respectively.  Then,  
by 
\begin{align*} 
&\sum_{i=1}^{a-1}m_i\\
&=\biggl(\frac{q(q+1)}{2}k-K-2(2 K+1-k)+(q+1)(r-1)+(q+2)(2 K-k+1)\biggr)a\\ 
&\quad +\biggl(\frac{a(a+1)}{2}+(3 K-k)a\biggr)d\\ 
&=\frac{a\bigl(a^2+(6K-k)a+K(5 K-3 k)-2 k-r(r+4 K-3 k)\bigr)}{2 k}\\
&\quad +\left(\frac{a+1}{2}+3 K-k\right)a d
\end{align*}
and  
\begin{align*} 
&\sum_{i=1}^{a-1}m_i^2\\
&=\biggl(
\frac{q(q+1)(2 q+1)}{6}k-9 K-4+4 k+(q+1)^2(r-1)+(q+2)^2(2 K-k+1)
\biggr)a^2\\ 
&\quad +\left(\frac{(a+1)(2 a+1)}{6}+(3 K-k)a+(K+1)(5 K+2)-k(k+1)\right)a d^2\\
&\quad +2 a d\biggl(-\frac{(K+1)(9 K+4)}{2}+k(k+1)+\frac{q k(q+1)\bigl((4 q-1)k+3\bigr)}{12}\\
&\qquad +(q+1)\biggl((r-K-1)\left(q k+\frac{r-K}{2}\right)+K\left(a+\frac{K+1}{2}\right)\biggr)\\
&\qquad +(q+2)\left(\frac{(a+2 K+1)(a+2 K+2)}{2}-\frac{(a+k)(a+k+1)}{2}\right)\biggr)\,,  
\end{align*}
we have 
\begin{align*}   
&s\bigl(a,a+(K+1)d,a+(K+2)d,\dots,a+k d\bigr)\\
&=\frac{1}{12 k^2}\biggl(
2 a^4+6(3 K-k)a^3+\bigl(30 K^2+24 K k-19 k^2+12 k-6(r+4 K-3 k)r\bigr)a^2\\ 
&\quad +\bigl(K^2(14 K+30 k)-4 K k(5 k-3)+12 k^2(k-1)+4 r^3+6(K-2 k)r^2\\ 
&\qquad -(24 K^2+18 K k-20 k^2+12 k)r\bigr)a -k^2\\ 
&\quad +\bigl(2 a^2+3(6 K-2 k+1)a+6 K(5 K+7)-6 k^2-6 k+13)\bigr)k^2 d^2\\
&\quad +\bigl(
4 a^3+3(12 K-4 k+1)a^2\\
&\qquad +\bigl(6 K(10 K+4 k+7)-(25 k^2-6 k-12)-6 r(r+4 K-3 k)\bigr)a\\ 
&\qquad +28 K^3-3 K^2(k-13)-K(7 k^2+9 k-12)+2 r^3-3 r^2(2 K+k+1)\\
&\qquad +r\bigl(-6 K(4 K-k+6)+7 k^2+9 k-12\bigr)
\bigr)k d
\biggr)\,. 
\end{align*} 

We also have 
$$
g\bigl(a,a+(K+1)d,a+(K+2)d,\dots,a+k d\bigr)
=\frac{a(a+K-r+k)}{k}+(a+2 K+1)d 
$$ 
and 
\begin{align*}  
&n\bigl(a,a+(K+1)d,a+(K+2)d,\dots,a+k d\bigr)\\
&=\frac{a^2+2(3 K-k)a+K(5 K-3 k)-k-r(r+4 K-3 k)}{2 k}\\
&\quad +\left(\frac{a+1}{2}+3 K-k\right)d\,. 
\end{align*}

In conclusion, we have the Sylvester sums.  

\begin{theorem}  
Let $a$ and $d$ be positive integers with $a\ge 2$ and $\gcd(a,d)=1$, $K$ and $k$ be positive integers with $(k-1)/2<K\le(2 k-2)/3$, and let $r=a+K\fl{(a+K)/k}k$. Assume that $\fl{(a+K)/k}\ge 2$. If $0\le r\le k-K-1$, then 
\begin{align*} 
&s\bigl(a,a+(K+1)d,a+(K+2)d,\dots,a+k d\bigr)\\
&=\frac{1}{12 k^2}\biggl(
2 a^4+6(3 K-k)a^3+\bigl(30 K^2-7 k^2+12 k-6(r+4 K-3 k)r\bigr)a^2\\ 
&\quad +\bigl(K^2(14 K+6 k)-4 K k(8 k-3)+12 k^2(k-1)+4 r^3+6(K-2 k)r^2\\ 
&\qquad -(24 K^2-6 K k-8 k^2+12 k)r\bigr)a -k^2\\ 
&\quad +\bigl(2 a^2+3(6 K-2 k+1)a+6 K(5 K+7)-(6 k^2+6 k-13)\bigr)k^2 d^2\\
&\quad +\bigl(
4 a^3+3(12 K-4 k+1)a^2\\
&\qquad +\bigl(6 K(10 K+7 k)-(13 k^2-6 k-12)-6 r(r+4 K-3 k)\bigr)a\\ 
&\qquad +28 K^3-K^2(27 k-39)-K(7 k^2+45 k-12)+6(k-1)k(k+2)
\bigr)k d
\biggr)\,.  
\end{align*} 
If $k-K\le r\le K$, then 
\begin{align*}   
&s\bigl(a,a+(K+1)d,a+(K+2)d,\dots,a+k d\bigr)\\
&=\frac{1}{12 k^2}\biggl(
2 a^4+6(3 K-k)a^3+\bigl(30 K^2+12 K k-19 k^2+24 k-6(r+4 K-5 k)r\bigr)a^2\\ 
&\quad +\bigl(K^2(14 K+18 k)-8 K k(4 k-3)+4 r^3+6(K-4 k)r^2\\ 
&\qquad -(24 K^2-6 K k-32 k^2+24 k)r\bigr)a -k^2\\ 
&\quad +\bigl(2 a^2+3(6 K-2 k+1)a+6 K(5 K+7)-6 k^2-6 k+13\bigr)k^2 d^2\\
&\quad +\bigl(
4 a^3+3(12 K-4 k+1)a^2\\
&\qquad +\bigl(6 K(10 K+2 k+7)-(25 k^2-18 k-12)-6 r(r+4 K-5 k)\bigr)a\\ 
&\qquad +28 K^3-K^2(9 k-39)-K(19 k^2+15 k-12)+2 r^3-3 r^2(2 K+3 k+1)\\
&\qquad +r\bigl(-K(24 K-18 k+36)+(19 k^2+15 k-12)\bigr)
\bigr)k d
\biggr)\,.  
\end{align*}
If $K+1\le r<k$, then 
\begin{align*}   
&s\bigl(a,a+(K+1)d,a+(K+2)d,\dots,a+k d\bigr)\\
&=\frac{1}{12 k^2}\biggl(
2 a^4+6(3 K-k)a^3+\bigl(30 K^2+24 K k-19 k^2+12 k-6(r+4 K-3 k)r\bigr)a^2\\ 
&\quad +\bigl(K^2(14 K+30 k)-4 K k(5 k-3)+12 k^2(k-1)+4 r^3+6(K-2 k)r^2\\ 
&\qquad -(24 K^2+18 K k-20 k^2+12 k)r\bigr)a -k^2\\ 
&\quad +\bigl(2 a^2+3(6 K-2 k+1)a+6 K(5 K+7)-6 k^2-6 k+13)\bigr)k^2 d^2\\
&\quad +\bigl(
4 a^3+3(12 K-4 k+1)a^2\\
&\qquad +\bigl(6 K(10 K+4 k+7)-(25 k^2-6 k-12)-6 r(r+4 K-3 k)\bigr)a\\ 
&\qquad +28 K^3-3 K^2(k-13)-K(7 k^2+9 k-12)+2 r^3-3 r^2(2 K+k+1)\\
&\qquad +r\bigl(-6 K(4 K-k+6)+7 k^2+9 k-12\bigr)
\bigr)k d
\biggr)\,. 
\end{align*} 
\label{th:ss456} 
\end{theorem}

Concerning Frobenius and Sylvester numbers, we have the following. $a,d,K,k,r$ are determined as in Theorem \ref{th:ss456}.   

\begin{theorem} 
Under the same conditions as in Theorem \ref{th:ss456}, we have  
\begin{align*}
&g\bigl(a,a+(K+1)d,a+(K+2)d,\dots,a+k d\bigr)\\
&=\begin{cases} 
\dfrac{a(a+K-r)}{k}+(a+2 K+1)d&\text{if $0\le r\le k-K-1$}\\
\dfrac{a(a+K-r+k)}{k}+(a+2 K+1)d&\text{if $k-K\le r<k$}\,. 
\end{cases}
\end{align*}  
\label{th:gg456} 
\end{theorem} 

\begin{theorem}  
Under the same conditions as in Theorem \ref{th:ss456}, if $0\le r\le k-K-1$, then 
\begin{align*}  
&n\bigl(a,a+(K+1)d,a+(K+2)d,\dots,a+k d\bigr)\\
&=\frac{a^2+2(3 K-k)a+K(5 K-7 k)+k(2 k-1)-r(r+4 K-3 k)}{2 k}\\
&\quad +\left(\frac{a+1}{2}+3 K-k\right)d\,. 
\end{align*} 
If $k-K\le r\le K$, then 
\begin{align*}  
&n\bigl(a,a+(K+1)d,a+(K+2)d,\dots,a+k d\bigr)\\
&=\frac{a^2+2(3 K-k)a+(K-r)(5 K-5 k+r)+k}{2 k}\\
&\quad +\left(\frac{a+1}{2}+3 K-k\right)d\,. 
\end{align*}
If $K+1\le r<k$, then 
\begin{align*}  
&n\bigl(a,a+(K+1)d,a+(K+2)d,\dots,a+k d\bigr)\\
&=\frac{a^2+2(3 K-k)a+K(5 K-3 k)-k-r(r+4 K-3 k)}{2 k}\\
&\quad +\left(\frac{a+1}{2}+3 K-k\right)d\,. 
\end{align*} 
\label{th:nn456} 
\end{theorem}

\subsection{Special patterns}   

For an integer $a\ge 2$, let us consider the sequence $a,a+4,a+5,a+6$. Then, we apply the Theorem \ref{th:ss456} as $K=3$, $k=6$ and $d=1$, and nonnegative integers $q$ and $r$ are determined by $a+3=6 q+r$ with $0\le r\le 5$. $q\ge 2$ implies that $a\ge 9$.  
When $r=0,1,2$, that is, $a\equiv 3,4,5\pmod 6$, we have  
\begin{align*}
&s(a,a+4,a+5,a+6)\\
&=\frac{1}{216}\bigl(a^4+21 a^3-3(r^2-6 r-66)a^2+(2 r^3-45 r^2+162 r+927)r\\
&\qquad +3(2 r^3-39 r^2+78 r+495)\,. 
\end{align*}
When $r=3$, that is, $a\equiv 0\pmod 6$, we have 
\begin{align*}
&s(a,a+4,a+5,a+6)\\
&=\frac{1}{216}\bigl(a^4+21 a^3-3(r^2-18 r-42)a^2+(2 r^3-81 r^2+774 r-153)r\\
&\qquad +3(2 r^3-75 r^2+762 r-793)\,. 
\end{align*}
When $r=4,5$, that is, $a\equiv 1,2\pmod 6$, we have  
\begin{align*}
&s(a,a+4,a+5,a+6)\\
&=\frac{1}{216}\bigl(a^4+21 a^3-3(r^2-6 r-66)a^2+(2 r^3-45 r^2+162 r+1143)r\\
&\qquad +3(2 r^3-39 r^2+78 r+999)\,. 
\end{align*} 
The final result also holds for $q=1$, that is, $a=8$.

\begin{Cor}  
For $a\ge 8$, we have 
$$
s(a,a+4,a+5,a+6)=\begin{cases}
\frac{a^4+21 a^3+261 a^2+1494 a+2808}{216}&\text{if $a\equiv 0\pmod 6$}\\
\frac{a^4+21 a^3+222 a^2+1199 a+2445}{216}&\text{if $a\equiv 1\pmod 6$}\\
\frac{a^4+21 a^3+213 a^2+1078 a+1992}{216}&\text{if $a\equiv 2\pmod 6$}\\
\frac{(a+3)(a^3+18 a^2+144 a+495)}{216}&\text{if $a\equiv 3\pmod 6$}\\
\frac{a^4+21 a^3+213 a^2+1046 a+1608}{216}&\text{if $a\equiv 4\pmod 6$}\\
\frac{a^4+21 a^3+222 a^2+1087 a+1533}{216}&\text{if $a\equiv 5\pmod 6$}\,. 
\end{cases}
$$ 
\label{cor:456} 
\end{Cor}

By applying Theorem \ref{th:gg456} as $K=3$, $k=6$ and $d=1$, for $r=0,1,2$, that is, $a\equiv 3,4,5\pmod 6$, we have 
$$ 
g(a,a+4,a+5,a+6)=\frac{a^2+(9-r)a+42}{6}\,, 
$$
and for $r=3,4,5$, that is, $a\equiv 0,1,2\pmod 6$, we have 
$$
g(a,a+4,a+5,a+6)=\frac{a^2+(15-r)a+42}{6}\,. 
$$ 
The case for $a=8$ is also valid.  
We can conclude that 
$$
g(a,a+4,a+5,a+6)=\begin{cases}
\frac{a^2+12 a+42}{6}&\text{if $a\equiv 0\pmod 6$}\\
\frac{a^2+11 a+42}{6}&\text{if $a\equiv 1\pmod 6$}\\
\frac{a^2+10 a+42}{6}&\text{if $a\equiv 2\pmod 6$}\\
\frac{a^2+9 a+42}{6}&\text{if $a\equiv 3\pmod 6$}\\
\frac{a^2+8 a+42}{6}&\text{if $a\equiv 4\pmod 6$}\\
\frac{a^2+7 a+42}{6}&\text{if $a\equiv 5\pmod 6$}\,. 
\end{cases}
$$ 
The coefficient of $a$ can be unified by using the floor function. 

\begin{Cor}  
For $a\ge 8$, we have 
$$
g(a,a+4,a+5,a+6)=\left(2+\fl{\frac{a}{6}}\right)a+7\,.
$$ 
\label{cor:456fn} 
\end{Cor}

By applying Theorem \ref{th:nn456}, for $r=0,1,2$, that is, $a\equiv 3,4,5\pmod 6$, we have 
$$
n(a,a+4,a+5,a+6)=\frac{a^2+12 a-(r+3)(r-9)}{12}\,, 
$$ 
for $r=3$, that is, $a\equiv 0\pmod 6$, we have  
$$
n(a,a+4,a+5,a+6)=\frac{a^2+12 a-r^2+18 r+3}{12}\,, 
$$ 
for $r=4,5$, that is, $a\equiv 1,2\pmod 6$, we have  
$$
n(a,a+4,a+5,a+6)=\frac{a^2+12 a-r^2+6 r+27}{12}\,.  
$$ 
Finally, we can check manually that the result is also valid for $a=4,5,6$.

\begin{Cor}  
For $a=4,5,6$ and $a\ge 8$, we have 
$$
n(a,a+4,a+5,a+6)=\begin{cases}
\frac{a^2+12 a+48}{12}&\text{if $a\equiv 0\pmod 6$}\\
\frac{a^2+12 a+35}{12}&\text{if $a\equiv 1\pmod 6$}\\
\frac{a^2+12 a+32}{12}&\text{if $a\equiv 2\pmod 6$}\\
\frac{a^2+12 a+27}{12}&\text{if $a\equiv 3\pmod 6$}\\
\frac{a^2+12 a+32}{12}&\text{if $a\equiv 4\pmod 6$}\\
\frac{a^2+12 a+35}{12}&\text{if $a\equiv 5\pmod 6$}\,. 
\end{cases}
$$ 
\label{cor:456sn} 
\end{Cor}

\noindent 
{\it Remark.}  
The coefficient of the constant term can be unified by using the floor function. 
\begin{align*}
&n(a,a+4,a+5,a+6)\\
&=\frac{a^2}{12}+a\\ 
&\quad +\frac{1}{12\cdot 15}\biggl(720+209 a-14\fl{\frac{a}{6}}-164\fl{\frac{a+1}{6}}-134\fl{\frac{a+2}{6}}\\
&\qquad -284\fl{\frac{a+3}{6}}-254\fl{\frac{a+4}{6}}-404\fl{\frac{a+5}{6}}\biggr)\\
&=\frac{a^2}{12}+\frac{389}{180}a+4\\
&\quad -\frac{1}{90}\biggl(7\fl{\frac{a}{6}}+82\fl{\frac{a+1}{6}}+67\fl{\frac{a+2}{6}}\\
&\qquad +142\fl{\frac{a+3}{6}}+127\fl{\frac{a+4}{6}}+202\fl{\frac{a+5}{6}}\biggr)\,. 
\end{align*}
\bigskip

For an integer $a\ge 2$, let us consider the sequence $a,a+5,a+6,a+7$. Then, $K=4$, $k=7$ and $d=1$ in Theorem \ref{th:ss456}.  Nonnegative integers $q$ and $r$ are determined by $a+4=7 q+r$ with $0\le r\le 6$. $q\ge 2$ implies that $a\ge 10$. 
When $r=0,1,2$, that is, $a\equiv 3,4,5\pmod 7$, we have  
\begin{align*}
&s(a,a+5,a+6,a+7)\\
&=\frac{1}{294}\bigl(a^4+29 a^3-(3 r^2-15 r-380)a^2+(2 r^3-51 r^2+151 r+2296)r\\
&\qquad +7(r^3-24 r^2+17 r+672)\,. 
\end{align*}
When $r=3,4$, that is, $a\equiv 6,0\pmod 7$, we have 
\begin{align*}
&s(a,a+5,a+6,a+7)\\
&=\frac{1}{294}\bigl(a^4+29 a^3-(3 r^2-57 r-296)a^2+(2 r^3-93 r^2+991 r+784)r\\
&\qquad +7(r^3-45 r^2+500 r-210)\,. 
\end{align*}
When $r=5,6$, that is, $a\equiv 1,2\pmod 7$, we have  
\begin{align*}
&s(a,a+5,a+6,a+7)\\
&=\frac{1}{294}\bigl(a^4+29 a^3-(3 r^2-15 r-422)a^2+(2 r^3-51 r^2+109 r+3346)r\\
&\qquad +7(r^3-24 r^2+17 r+1386)\,. 
\end{align*} 
We can check manually that the result is also valid for $a=5,6,7$.   

\begin{Cor}  
For $a=5,6,7$ and $a\ge 10$, we have 
$$
s(a,a+5,a+6,a+7)=\begin{cases}
\frac{a^4+29 a^3+476 a^2+3388 a+7938}{294}&\text{if $a\equiv 0\pmod 7$}\\
\frac{a^4+29 a^3+422 a^2+2866 a+6972}{294}&\text{if $a\equiv 1\pmod 7$}\\
\frac{a^4+29 a^3+404 a^2+2596 a+5880}{294}&\text{if $a\equiv 2\pmod 7$}\\
\frac{a^4+29 a^3+380 a^2+2296 a+4704}{294}&\text{if $a\equiv 3\pmod 7$}\\
\frac{a^4+29 a^3+392 a^2+2398 a+4662}{294}&\text{if $a\equiv 4\pmod 7$}\\
\frac{a^4+29 a^3+398 a^2+2410 a+4326}{294}&\text{if $a\equiv 5\pmod 7$}\\
\frac{a^4+29 a^3+440 a^2+2974 a+6384}{294}&\text{if $a\equiv 6\pmod 7$}\,. 
\end{cases}
$$ 
\label{cor:567} 
\end{Cor}

By applying Theorem \ref{th:gg456} as $K=4$, $k=7$ and $d=1$, for $r=0,1,2$, that is, $a\equiv 3,4,5\pmod 7$, we have 
$$ 
g(a,a+5,a+6,a+7)=\frac{a^2+(11-r)a+63}{7}\,, 
$$
and for $r=3,4,5,6$, that is, $a\equiv 6,0,1,2\pmod 7$, we have 
$$
g(a,a+5,a+6,a+7)=\frac{a^2+(18-r)a+63}{7}\,. 
$$ 
The cases for $a=5,6,7,8$ are also valid.  
The coefficient of $a$ can be unified by using the floor function. 

\begin{Cor}  
For $a=5,6,7,8$ and $a\ge 10$, we have 
$$
g(a,a+5,a+6,a+7)=\left(2+\fl{\frac{a+1}{7}}\right)a+9\,.
$$ 
\label{cor:567fn} 
\end{Cor}

By applying Theorem \ref{th:nn456}, for $r=0,1,2$, that is, $a\equiv 3,4,5\pmod 7$, we have 
$$
n(a,a+5,a+6,a+7)=\frac{a^2+17 a-r^2+5 r+52}{14}\,, 
$$ 
for $r=3,4$, that is, $a\equiv 6,0\pmod 7$, we have  
$$
n(a,a+5,a+6,a+7)=\frac{a^2+17 a-r^2+19 r+24}{14}\,, 
$$ 
for $r=5,6$, that is, $a\equiv 1,2\pmod 7$, we have  
$$
n(a,a+5,a+6,a+7)=\frac{a^2+17 a-r^2+5 r+66}{14}\,.  
$$ 
Finally, we can check manually that the result is also valid for $a=5,6,7$.

\begin{Cor}  
For $a=5,6,7$ and $a\ge 10$, we have 
$$
n(a,a+4,a+5,a+6)=\begin{cases}
\frac{a^2+17 a+84}{14}&\text{if $a\equiv 0\pmod 7$}\\
\frac{a^2+17 a+66}{14}&\text{if $a\equiv 1\pmod 7$}\\
\frac{a^2+17 a+60}{14}&\text{if $a\equiv 2\pmod 7$}\\
\frac{a^2+17 a+52}{14}&\text{if $a\equiv 3\pmod 7$}\\
\frac{a^2+17 a+56}{14}&\text{if $a\equiv 4\pmod 7$}\\
\frac{a^2+17 a+58}{14}&\text{if $a\equiv 5\pmod 7$}\\
\frac{a^2+17 a+72}{14}&\text{if $a\equiv 6\pmod 7$}\,. 
\end{cases}
$$ 
\label{cor:567sn} 
\end{Cor}

\section{Comments}  

Similarly, we can consider the sequence $a,a+(K+1)d,a+(K+2)d,\dots,a+k d$ when $(2 k-2)/3<K\le(3 k-3)/4$. Then, as a special case, we can get Frobenius number, Sylvester number and sum for $a,a+6,a+7,a+8$ and so on. After that, we may continue to consider the cases $(3 k-3)/4<K\le(4 k-4)/5$, $(4 k-4)/5<K\le(5 k-5)/6$, $\dots$. However, the situation becomes more and more complicated. Is there any more convenient method to find their Sylvester sums? 

The approaches in \cite{ro79,PS90} may apply to any almost arithmetic sequence, but they both have an extra burden: the require the pre-computation of a couple of constants depending on the sequence. Particularizing their results to ours would be of some interest.

\end{document}